\documentclass[12pt]{article}
\usepackage[english]{babel}
\usepackage[utf8]{inputenc}
\usepackage{geometry}
\usepackage{yfonts}
\usepackage{xfrac}
\usepackage[c]{esvect}
\usepackage{setspace}
\usepackage{graphicx}
\usepackage{indentfirst} 
\usepackage{amsmath,amsfonts,amsthm,amscd,upref,amstext}	
\usepackage{amssymb}
\usepackage[titletoc]{appendix}

\counterwithin{equation}{section}

\title{The relationship between some special conditions of young functions and the validity of generalized $L^p$ estimate for the Poisson equations in a unit ball}
\author{Tianxiao Hu\footnote{Peking University, E-mail: txhu@pku.edu.cn}}
\date{}

\begin{document}
\maketitle
\doublespacing
\newtheorem{definition}{Definition}[section]
\newtheorem{theorem}[definition]{Theorem}
\newtheorem{remark}[definition]{Remark}
\newtheorem{lemma}{Lemma}[subsection]

\abstract{In this paper, we consider the domain is in $B_1$, and we will show the relationship between  the global $\Delta_2$ and $\nabla_2$ conditions of young functions and the validity of generalized $L^p$ estimate for the Poisson equations.}

\section{INTRODUCTION}
In the 20th century, Sobolev spaces were commonly used in a wide variety of fields of pdes. Since the 1960s, many spaces of functions wider than Sobolev spaces were created for meeting the needs of various practical problems. Orlicz spaces (see Definition (1.5))have been studied as the generalization of Sobolev spaces since they were introduced by Orlicz \cite{Orlicz}. The theory of Orlicz spaces plays a crucial role in many fields of mathematics including geometric, probability, stochastic, Fourier analysis and partial differential equations (see \cite{rao2002applications}). 

In this paper, we are concerned about the following the Dirichlet problem  in a unit ball:
\begin{equation}
\begin{cases}  \label{Poisson} 
 -\Delta u=f  \quad \quad   in  \quad B_1, \\ 

u |_{\partial B_1}=0.
\end{cases}
\end{equation}

where the dimension $n \geq 2$. Our main purpose is to study what are the optimal conditions on those Young functions $\phi$ that satisfy the estimate
\begin{equation} \label{estimate}
\int_{B_1} \phi\left(\left|D^2 u\right|\right) dx \leq C \int_{B_1} \phi(|f|) d x,
\end{equation}
for all pairs $(u, f)$ satisfying Poisson equation \eqref{Poisson} where $C$ is a positive constant independent of $u$ and $f$.
Indeed, if $\phi(t)=|t|^p$, \eqref{estimate} is reduced to the classical $L^p$ estimate (see \cite{chen1998second,gilbarg1977elliptic} where  the domain is $\mathbb{R}^n$ ). That's the reason why we call it `generalized $L^p$ estimate' in the title.

In the case where the domain is $\mathbb{R}^n$, there are other previous related works.  Wang \cite{wang2003geometric} gave a new proof of local $L^p$ estimates for the Poisson and heat equation by a geometric approach, in which the Hardy-Littlewood maximal function, modified Vitali covering lemma, and compactness method are used. By employing the same techniques as in \cite{wang2003geometric}, Jia, Li and Wang \cite{jia2007regularity} generalized local estimates in $L^p$ space to Orlicz spaces for the Poisson equation when $\phi \in \Delta_2 \cap \nabla_2$ (see Definition (\ref{Young1}) and (\ref{Young2})). Since $\phi$ is not certain to be a polynomial, which leads to the failure of the normalization, the authors in \cite{jia2007regularity} first assume that $u \in C_0^{\infty}(\Omega)$ and then use an interpolation inequality to obtain the result. Acerbi and Mingione \cite{acerbi2007gradient} obtained local $L^q, q \geq p$, gradient estimates for the degenerate parabolic $p$-Laplacian systems. There they invent a new iteration-covering approach, which is completely free from harmonic analysis. Wang, Yao, Zhou and Jia \cite{wang2009optimal} simplified the iteration-covering procedure used in \cite{acerbi2007gradient} and extended it to the whole space.

In this paper, we use the following notations. Let $B_r=\left\{y \in \mathbb{R}^n:|y|<r\right\}$ be an open ball in $\mathbb{R}^n$ with center 0 and radius $r>0$, and $B_r(x)=B_r+x$. We denote
$$
\left|D^2 u\right|=\sum_{i, j=1, \ldots, n}\left|D_{x_i x_j} u\right|
\quad and\quad 
\left\|D^2 u\right\|_{L^p\left(\Omega\right)}=\sum_{i, j=1, \ldots, n}\left\|D_{x_i x_j} u\right\|_{L^p\left(\Omega\right)},
$$
where
$$
\left\|D_{x_i x_j} u\right\|_{L^p\left(\Omega\right)}=\left(\int_{\Omega}\left|D_{x_i x_j} u\right|^p d x\right)^{1 / p} \text { for } p>1 .
$$
In addition, we denote the set
$$
\Phi=\left\{\phi:[0,+\infty) \longrightarrow[0,+\infty)\quad |\quad \phi \quad is \quad increasing \quad and \quad convex\right\}.   
$$

\begin{definition} 
A function $\phi \in \Phi$ is said to be a Young function if
$$
\lim _{t \rightarrow 0+} \frac{\phi(t)}{t}=\lim _{t \rightarrow+\infty} \frac{t}{\phi(t)}=0.
$$
\end{definition}

\begin{definition}
A Young function $\phi \in \Phi$ is said to satisfy the global $\nabla_2$ condition, denoted by $\phi \in \nabla_2$, if there exists a number $a>1$ such that for every $t>0$,
$$
\phi(t) \leq \frac{\phi(a t)}{2 a} .
$$
\end{definition}

We can verify that if $\phi \in \nabla_2$, then $\phi$ satisfies for $0<\theta_1 \leq 1$,
\begin{equation}  \label{Young1}
\phi\left(\theta_1 t\right) \leq 2 a \theta_1^{\alpha_2} \phi(t),
\end{equation}
where $\alpha_2=\log _a 2+1 (> 1)$.

\begin{definition}
A Young function $\phi \in \Phi$ is said to satisfy the global $\Delta_2$ condition, denoted by $\phi \in \Delta_2$, if there exists a positive constant $K$ such that for every $t>0$,
$$
\phi(2 t) \leq K \phi(t) .
$$
Obviously, $K \geq 2$. Also if $K=2$, $\phi$ is linear so it cannot satisfy the global $\nabla_2$ condition. 
\end{definition}

It is easy to check if $\phi \in \Delta_2$,  for $1 \leq \theta_2<\infty$ $\phi$ satisfies
\begin{equation} \label{Young2}
\phi\left(\theta_2 t\right) \leq K \theta_2^{\alpha_1} \phi(t),
\end{equation}
where $\alpha_1=\log _2 K (\geq 1)$. When young function $\phi$ satisfy the global $\nabla_2$ and $\Delta_2$ condition at the same time, $K>2$ and $\alpha_1>1$.

\begin{remark}
    The global $\Delta_2 \cap \nabla_2$ condition makes the function grow moderately. For example, $\phi(t)=|t|^\alpha(1+|\log | t||)$ for $\alpha>1$ satisfies the global $\Delta_2 \cap \nabla_2$ condition.
\end{remark}

\begin{definition} 
 Let $\phi$ be a Young function. Then the Orlicz class $K^\phi\left( B_1\right)$ is the set of all measurable functions $g: B_1 \rightarrow \mathbb{R}$ satisfying
$$
\int_{B_1} \phi(|g|) d x<\infty .
$$
The Orlicz space $L^\phi\left( B_1\right)$ is the linear hull of $K^\phi\left( B_1\right)$.
\end{definition}

\begin{remark} \label{dense property}
    In general, $K^\phi \subset L^\phi$. However, if $\phi$ satisfies the global $\Delta_2$ condition, then $K^\phi=L^\phi$ and $C_0^{\infty}$ is dense in $L^\phi$ (see \cite{adams2003sobolev}). 
\end{remark}

\begin{remark} \label{another integral form}
    If $g \in L^\phi\left( B_1\right)$, then $\int_{ B_1} \phi(|g|) d x$ can be rewritten in an integral form as
\begin{equation}
\int_{ B_1}\phi(|g|) d x=\int_0^{\infty}\left|\left\{x \in  B_1:|g|>\lambda\right\}\right| d[\phi(\lambda)] .
\end{equation}
\end{remark}

Now let us state the main result of this work:

\begin{theorem} \label{main result}
Assume that $\phi \in \Phi$. 

(1) If estimate \eqref{estimate} holds for every pair $(u, f) \in$ $C_0^{\infty}\left(B_1\right) \times C_0^{\infty}\left(B_1\right)$ satisfying the Dirichlet problem \eqref{Poisson} and $D^2 u \in L^\phi\left(B_1\right)$, then $\phi \in \Delta_2 \cap \nabla_2$. 

(2) On the contrary, if $\phi \in \Delta_2 \cap \nabla_2$, then for every $f \in L^\phi\left(B_1\right)$, there is a solution $u \in W_{l o c}^{2,1}\left(B_1\right)$ and $u |_{\partial B_1}=0$, satisfying estimate \eqref{estimate}.
\end{theorem}

\section{PROOF OF THE MAIN RESULT}
\subsection{Proof for (1) of Theorem 1.7}
In this subsection we show that $\phi \in \triangle_2 \cap \nabla_2$ if estimate (\ref{estimate}) is true.
\subsubsection{$\phi$ satisfies the global $\nabla_2$ condition.}
Now we consider the special case in (\ref{Poisson}) when
$$
f_t(x)=t \eta,
$$
where $t$ is a positive parameter, $\eta \in C_0^{\infty}\left(B_1\right)$ is a cutoff function satisfying
\begin{equation} \label{eta def}
0 \leq \eta(x) \leq 1, \quad \eta(x) \equiv 1 \quad in \quad B_{\frac{1}{6\sqrt{n}}} , \quad\eta(x)=0 \quad in \quad B_{\frac{1}{6\sqrt{n}}} / B_{\frac{1}{12\sqrt{n}}} .
\end{equation}
Therefore equation \ref{Poisson} has a solution
\begin{equation} \label{eta's solution}
    u_t(x)=\int_{B_1} \Gamma(x-\xi) f_t(\xi) d \xi,
\end{equation}

where
$$
\Gamma(x)= \begin{cases}\frac{1}{n(n-2) w_n} \frac{1}{|x|^{n-2}} & (n>2), \\ -\frac{1}{2 \pi} \ln |x| & (n=2)\end{cases}
$$
is the fundamental solution of $-\triangle$.

It follows from (\ref{estimate}) and (\ref{eta def}) that
\begin{equation} \label{ut estimate}
\int_{B_1} \phi\left(\left|D^2 u_t\right|\right) d x \leq C \int_{B_1} \phi\left(\left|f_t\right|\right) dx 
=  C \int_{B_{\frac{1}{6\sqrt{n}}}} \phi\left(\left|f_t\right|\right) dx
\leq C \int_{B_{\frac{1}{6\sqrt{n}}}} \phi\left(t\right) dx \leq C \phi(t) . 
\end{equation}

We know from (\ref{eta's solution}) that when $|x|>\frac{1}{6\sqrt{n}}$, the following integral has no singular points so it is valid:
$$
D_{x_i x_i} u_t(x)=\int_{B_\frac{1}{6\sqrt{n}}} \frac{1}{w_n|x-\xi|^n}\left[\frac{\left(x_i-\xi_i\right)^2}{|x-\xi|^2}-\frac{1}{n}\right] f_t(\xi) d \xi.
$$

We define
$$
D:=\left\{x \in B_1: |x|\geq\frac{1}{2}+\frac{7}{12\sqrt{n}} \quad  and  \quad |x_1| \geq \frac{4}{3}|x|\right\}.
$$
When $x \in D, \xi \in B_{\frac{1}{6\sqrt{n}}} / B_{\frac{1}{12\sqrt{n}}}$, we can compute that
$$
\frac{\left|x_1-\xi_1\right|}{|x-\xi|} \geq \frac{|x_1|-\frac{1}{6\sqrt{n}}}{|x|+\frac{1}{6\sqrt{n}}} \geq \frac{\frac{4}{3}|x| -\frac{1}{6\sqrt{n}}}{|x|+\frac{1}{6\sqrt{n}}} \geq \frac{1}{\sqrt{n}} .
$$
When $x \in D, \xi \in B_\frac{1}{12\sqrt{n}}$,
$$
\frac{\left|x_1-\xi_1\right|}{|x-\xi|} \geq \frac{\left|x_1\right|-\frac{1}{12\sqrt{n}}}{|x|+\frac{1}{12\sqrt{n}}} \geq \frac{\frac{4}{3}|x| -\frac{1}{12\sqrt{n}}}{|x|+\frac{1}{12\sqrt{n}}} \geq \frac{7}{6} \frac{1}{\sqrt{n}},
$$
and
$$
|x-\xi| \leq|x|+|\xi| \leq|x|+\frac{1}{2}|x| \leq \frac{3}{2}|x| .
$$
Therefore, for $x \in D$ we conclude that
$$
\begin{aligned}
D_{x_1 x_1} u_t(x) & =t \int_{B_\frac{1}{6\sqrt{n}}} \frac{1}{w_n|x-\xi|^n}\left[\frac{\left(x_1-\xi_1\right)^2}{|x-\xi|^2}-\frac{1}{n}\right] \eta d \xi \\
& \geq t \frac{2^n}{3^n} \frac{1}{w_n|x|^n} \int_{B_\frac{1}{12\sqrt{n}}}\left[\left(\frac{7}{6}\right)^2 \frac{1}{n}-\frac{1}{n}\right] d \xi \\
& \geq \frac{t}{36^n n^{\frac{n}{2}} n^2}|x|^{-n} \geq \frac{t}{M^n n^2}|x|^{-n}.
\end{aligned}
$$
where $M=36\sqrt{n} $. Recalling estimate (\ref{ut estimate}) we find that
$$
\begin{aligned}
\int_D \phi\left(\frac{t}{M^n n^2}|x|^{-n}\right) dx & \leq C    \int_D \phi\left(D_{x_1 x_1} u_t(x)\right) dx  
  \leq C  \int_D \phi\left(|D^2 u_t\right|) dx   \\
  & \leq C  \int_{B_1} \phi\left(|D^2 u_t\right|) dx   \leq C \phi(t).
\end{aligned}
$$

which implies that
$$
\int_{\frac{1}{2}+\frac{7}{12\sqrt{n}}}^{1} \phi\left(\frac{t}{M^n n^2} r^{-n}\right) r^{n-1} dr \int_{\left|\cos \theta_1\right|>\frac{4}{3}} d \omega \leq C \phi(t)
$$
By changing the variable, let $s=\frac{t}{M^n n^2} r^{-n} $, it is easy to find that for $t>0$,
$$
\int_{\alpha_1 t}^{\alpha_2 t} \frac{\phi(s)}{s^2} ds \leq \frac{C \phi(t)}{t},
$$
where $\alpha_1=M^{-n} n^{-2}$ and $\alpha_2=(\frac{1}{2}+\frac{7}{12\sqrt{n}})^{-n}M^{-n} n^{-2}$. 

We should note that $\phi(0)=0$ due to the definition of Young function and the countinuity of convex function. Combined with the convexity of $\phi$, if $0< {t_1}\leq {t_2}$, we immediately know that $\frac{\phi(t)}{t}$ is an increasing function from
$$
\frac{\phi(t_1)-\phi(0)}{t_1-0} \leq \frac{\phi(t_2)-\phi(0)}{t_2-0}.
$$
Then
$$
\frac{\phi(t)}{t} \geq \frac{1}{C}\int_{\alpha_1 t}^{\alpha_2 t} \frac{\phi(s)}{s^2} ds \geq \frac{1}{C}\frac{\phi(\alpha_1t)}{\alpha_1t}\int_{\alpha_1 t}^{\alpha_2 t} \frac{1}{s}ds = C\frac{\phi(\alpha_1t)}{\alpha_1t}
\geq C\frac{\phi(\epsilon t)}{\epsilon t},
$$
where we can choose $\epsilon$ small enough, until $\phi$ satisfies the global $\nabla_2$  condition.

\subsubsection{$\phi$ satisfies the global $\Delta_2$ condition.}
Define two constants,
$$
C_1=\max _{x \in B_1}|\Delta \eta| 
\quad and \quad
C_2=\max _{x \in {B_1}}\left\{\left|D^2 \eta\right|=\sum_{i, j=1, \ldots, n}\left|D_{x_i x_j} \eta\right|\right\},
$$
where $\eta(x) \in C_0^{\infty}\left(B_1\right)$ is a cutoff function defined in (\ref{eta def}). It is easy to see that $C_2>C_1$. 

Now we substitute the special pairs into (\ref{estimate}),
$$
u_t(x)=\frac{t \eta(x)}{C_1} \quad \text { and } \quad f_t(x)=-\frac{t \Delta \eta(x)}{C_1},
$$
where $t>0$.
From the proof of section 2.1.1, we know that $\phi \in \nabla_2$ if estimate (1.2) is true. Set
$$
C_3=\frac{C_1+C_2}{2}, \quad \gamma=\frac{C_3}{C_1} .
$$
It is obvious that $\gamma>1$. Then from (\ref{estimate}) and (\ref{Young1}) we obtain
$$
\begin{aligned}
\phi(\gamma t)\left|\left\{x \in B_1:\left|D^2 \eta\right|>C_3\right\}\right| & =\phi(\gamma t)\left|\left\{x \in 
 B_1:\left|D^2 u_t\right|>\gamma t\right\}\right| \\
& \leq \int_{B_1} \phi\left(\left|D^2 u_t\right|\right) d x \\
& \leq C \int_{B_1} \phi\left(\left|f_t\right|\right) d x \\
& \leq C \phi(t) \int_{B_1} 2 a\left(\frac{|\Delta \eta|}{C_1}\right)^{\alpha_2} d x \\
& \leq C \phi(t) .
\end{aligned}
$$
Therefore, we conclude that
$$
\phi(\gamma t) \leq C \phi(t),
$$
which implies that
$$
\phi(2 t) \leq C \phi(t) .
$$
This completes our proof.

\subsection{Proof for (2) of Theorem \ref{main result}}
\subsubsection{In the case when $f \in C_0^{\infty}\left(B_1\right)$.}
According to Remark \ref{dense property}, when $\phi$ satisfies the global $\Delta_2$ condition,  $C_0^{\infty}\left(B_1\right)$ is dense in Orlicz space, so we first consider the result when $f \in C_0^{\infty}\left(B_1\right)$. This moment by the classical theory, there is only one solution $u \in C_0^{\infty}\left(B_1\right)$ satisfying (\ref{Poisson}). It follows that $ D^2 u \in C_0^{\infty}\left(B_1\right)$ so that $ D^2 u \in L^\phi$.

Also because Remark \ref{another integral form}, we can first rewrite the integral like this:
$$
\int_{B_1} \phi\left(\left|D^2 u\right|\right) dx=  \int_0^{\infty}\left|\left\{x \in B_1:\left|D^2 u\right|>\lambda\right\}\right| d\left[\phi\left(\lambda\right)\right].
$$

We choose to use a method motivated by the iteration-covering procedure in \cite{acerbi2007gradient,wang2009optimal}. First, we need to estimate the measure of $\left|\{x \in B_1:|D^2 u|>\lambda\}\right|$. In order to use $L^p$ estimate, we choose a fix constant $p$ with 
\begin{equation} \label{p definition}
    1<p<\alpha_2
\end{equation} 
where $\alpha_2$ is defined in (\ref{Young2}). The reason of choice of p will be discussed later. And  we denote
\begin{equation} \label{E definition}
    E^p=\int_{B_1}\left|D^2 u\right|^p d x+M^p \int_{B_1}|f|^p dx,
\end{equation}
while $M>1$ is a large enough constant which will be determined later. Set
\begin{equation} \label{ulambda definition}
    u_\lambda= \frac{u}{E\lambda}  \quad and \quad  f_\lambda=\frac{f}{E \lambda},
\end{equation}
for any $\lambda>0$. $u_\lambda$ is still the solution of (\ref{Poisson}) with $f_\lambda$ replacing $f$. 

In addition, for any domain $B$ in $\mathbb{R}^n$, we write
$$
J_\lambda[B]=\int_B \hspace{-1.20em}- \left|D^2 u_\lambda\right|^p d x+M^p \int_B \hspace{-1.20em}-\left|f_\lambda\right|^p d x
$$
and
$$
E_\lambda(1)=\left\{x \in B_1:\left|D^2 u_\lambda\right|>1\right\} =\left\{x \in B_1:\left|D^2 u\right|>\lambda E\right\} .
$$
Since $\left|D^2 u_\lambda(x)\right| \leq 1$ for $x \in B_1 \backslash E_\lambda(1)$, we focus our attention on the level set $E_\lambda(1)$. Next, we will decompose the level set $E_\lambda(1)$.

\begin{lemma}
For $\forall \lambda>0$, there exists a family of disjoint balls $\left\{B_{\rho_i}\left(x_i\right)\right\}_{i \geq 1}$ with $x_i \in E_\lambda(1)$ and $\rho_i=\rho\left(x_i, \lambda\right)>0$ such that
\begin{equation}  \label{B rho estimate}
    J_\lambda\left[B_{\rho_i}\left(x_i\right)\right]=1, \quad J_\lambda\left[B_\rho\left(x_i\right)\right]<1 \quad for \quad \forall   \rho >\rho_i,
\end{equation}
and
$$
E_\lambda(1) \subset \bigcup_{i \geq 1} B_{5 \rho_i}\left(x_i\right) \cup \text { negligible set. }
$$
\end{lemma}
\begin{proof}
    Fix $\forall x \in B_1$ and fixed $\lambda$,  there exists $\rho_0=\rho_0(\lambda)>0$ satisfying $\lambda^p\left|B_{\rho_0}(x)\right|=1$. 
    When $\rho > \rho_0$, 
$$
J_\lambda\left[B_\rho(x)\right] \leq \frac{1}{\left|B_\rho(x)\right|}\left[\int_{B_1}\left|D^2 u_\lambda\right|^p d x+M^p \int_{B_1}\left|f_\lambda\right|^p d x\right] \leq \frac{1}{\lambda^p\left|B_\rho(x)\right| } < 1 .
$$
Thus we conclude that
\begin{equation} \label{J estimate}
    \sup _{x \in B_1} \sup _{\rho \geq \rho_0} J_\lambda\left[B_\rho(x)\right] \leq 1.
\end{equation}

For a.e. $x \in E_\lambda(1)$, by Lebesgue's differentiation theorem we know that
$$
\lim _{\rho \rightarrow 0} J_\lambda\left[B_\rho(x)\right] = \left|D^2 u_\lambda (x)\right|>1,
$$
which implies that there exists some $\rho>0$ satisfying
$
J_\lambda\left[B_\rho(x)\right]>1 .
$
Compared with (\ref{J estimate}), one can select $\rho_x \in\left(0, \rho_0\right]$ such that
$$
J_\lambda\left[B_{\rho_x}(x)\right]=1, \quad J_\lambda\left[B_\rho(x)\right]<1 \text { for any } \rho>\rho_x .   
$$

It follows from the argument above that for a.e. $x \in E_\lambda(1)$ there exists a ball $B_{\rho_x}(x)$ constructed as above. Therefore, applying Vitali's covering lemma, we can find a family of disjoint balls $\left\{B_{\rho_i}\left(x_i\right)\right\}_{i \geq 1}$ such that the results of the lemma hold. This completes our proof.
\end{proof}

Now We can use the union of the balls covering the level set $E_{\lambda}(1)$. So for every fixed ball $\left\{B_{\rho_i}\left(x_i\right)\right\}$, we need an estimate of its measure.

\begin{lemma}
   Under the same hypotheses and results as those in Lemma 2.2.1, we have
$$
\begin{aligned}
\left|B_{\rho_i}\left(x_i\right)\right| \leq \frac{2^{p-1}}{2^{p-1}-1} \left( \int_{{x \in B_{\rho_i}(x_i):\left|D^2 u_\lambda\right|>1/2}}\left|D^2 u_\lambda\right|^p dx   \right.\\ \left.
 +M^p \int_{\left\{x \in B_{\rho_i}(x_i):\left|f_\lambda\right|>1/(2M)\right\}}\left|f_\lambda\right|^p dx \right) .
\end{aligned}
$$

\end{lemma}

\begin{proof}
From (\ref{B rho estimate}) in the lemma above we see that
$$
\left|B_{\rho_i}\left(x_i\right)\right|=\int_{B_{\rho_i}\left(x_i\right)}\left|D^2 u_\lambda\right|^p d x+M^p \int_{B_{\rho_i}\left(x_i\right)}\left|f_\lambda\right|^p d x .
$$
Therefore, by splitting the two integrals above as follows we have
$$
\begin{aligned}
\left|B_{\rho_i}\left(x_i\right)\right|        \leq  
&\int_{\left\{x \in B_{\rho_i}(x_i):\left|D^2 u_\lambda\right|>1/2\right\}}\left|D^2 u_\lambda\right|^p dx+(\frac{1}{2})^p\left|B_{\rho_i}(x_i)\right|   \\
 +M^p &\int_{\left\{x \in B_{\rho_i}\left(x_i\right):\left|f_\lambda\right|>1 /(2M)\right\}}\left|f_\lambda\right|^p d x+(\frac{1}{2})^p\left|B_{\rho_i}\left(x_i\right)\right|  .
\end{aligned}
$$
After transposition, the proof is completed.
\end{proof}

Now the problem becomes to a new one: after replacing the measure of level set by the $L^p$ estimate of $D^2 u_\lambda$ and $f_\lambda$,  we wish the integral in remark (\ref{another integral form}) could be controlled by the RHS of estimate (\ref{estimate}), so we need the following lemma:
\begin{lemma}
    If $\phi \in \Phi$ satisfies the global $\triangle_2 \cap \nabla_2$ condition, and $g \in L^\phi$, then for any $b_1, b_2>0$ we have
$$
\int_0^{\infty} \frac{1}{\mu^p}\left\{\int_{\left\{x \in B_1:|g|>b_1 \mu\right\}}|g|^p d x\right\} d\left[\phi\left(b_2 \mu\right)\right] \leq C\left(b_1, b_2, \phi\right) \int_{B_1} \phi(|g|) d x .
$$
\end{lemma}

\begin{proof}
Interchanging the order of integration and integrating by parts, we deduce that
$$
\begin{aligned}
I & =: \int_{B_1}|g|^p\left\{\int_0^{\frac{|g|}{b_1}} \frac{d\left[\phi\left(b_2 \mu\right)\right]}{\mu^p}\right\} dx \\
& \leq \int_{B_1}|g|^p\left\{\frac{\phi\left(b_2\frac{|g|}{b_1}\right)}{\left(\frac{|g|}{b_1}\right)^p}+p \int_0^{\frac{|g|}{b_1}} \frac{\phi\left(b_2 \mu\right)}{\mu^{p+1}} d \mu\right\} d x,
\end{aligned}
$$
and it follows from (\ref{Young1}), (\ref{Young2}) and (\ref{p definition}) that
$$
\begin{aligned}
I & \leq C \int_{B_1} \phi(|g|) d x+2 a p b_1^{\alpha_2} \int_{B_1} \phi\left(b_2\frac{|g|}{b_1}\right)|g|^{p-\alpha_2}\left\{\int_0^{\frac{|g|}{b_1}} \frac{1}{\mu^{p+1-\alpha_2}} d \mu\right\} d x \\
& \leq C \int_{B_1} \phi(|g|) d x .
\end{aligned}
$$
Here we use the condition $p<\alpha_2$, and it explain the reason for choice of $p$.
\end{proof}

Now we started prove the (2) in Theorem \ref{main result} when $f \in C_0^{\infty} (B_1)$.

\begin{proof}
    
Fix $i \geq 1 $ and $ B_{\rho_i}$. By Lemma 2.2.1, 
\begin{equation} \label{ulambda flambda}
    \int_{B_{10 \rho_i}\left(x_i\right)} \hspace{-3.90em}- \quad  \quad\left|D^2 u_\lambda\right|^p dx < 1 \text { and } \int_{B_{10 \rho_i}\left(x_i\right)} \hspace{-3.90em}- \quad 
 \quad \left|f_\lambda\right|^p dx 
 < \frac{1}{M^p} .
\end{equation}

Now let $v_\lambda$ satisfy the following boundary problem (the solution exists for sure)
$$
\left\{\begin{aligned}
-\Delta v_\lambda=0 & \text { in } B_{10 \rho_i}\left(x_i\right), \\
v_\lambda=u_\lambda & \text { on } \partial B_{10 \rho_i}\left(x_i\right) .
\end{aligned}\right.
$$
Let $w_\lambda=u_\lambda-v_\lambda$. Then $w$ satisfies 
$$
\left\{\begin{aligned}
-\Delta w_\lambda=f_\lambda & \text { in } B_{10 \rho_i}\left(x_i\right), \\
w_\lambda=0 & \text { on } \partial B_{10 \rho_i}\left(x_i\right) .
\end{aligned}\right.
$$
Thus from the elementary $L^p$ estimates and (\ref{ulambda flambda}) we find that
\begin{equation} \label{w estimate}
\int_{B_{10 \rho_i}\left(x_i\right)} \hspace{-3.90em}- \quad  \quad\left|D^2 w_\lambda\right|^p d x \leq C \int_{B_{10 \rho_i}\left(x_i\right)} \hspace{-3.90em}- \quad  \quad\left|f_\lambda\right|^p d x \leq \frac{C}{M^p} .    
\end{equation}
Note that $v_\lambda=u_\lambda-w_\lambda$,
$$ 
    \int_{B_{10 \rho_i}\left(x_i\right)} \hspace{-3.90em}- \quad  \quad\left|D^2 v_\lambda\right|^p dx 
    \leq 2^{p-1} \left\{
    \int_{B_{10 \rho_i}\left(x_i\right)} \hspace{-3.90em}- \quad  \quad\left|D^2 w_\lambda\right|^p dx +
    \int_{B_{10 \rho_i}\left(x_i\right)} \hspace{-3.90em}- \quad  \quad\left|D^2 u_\lambda\right|^p\right\} d x \leq C,
$$
and due to the $W_{\text {loc }}^{2, \infty}$ regularity,
\begin{equation} \label{v D^2 estimate}
\sup _{B_{5 \rho_i}\left(x_i\right)}\left|D^2 v_\lambda\right| \leq N_1,
\end{equation}
where $N_1>1$ only depends on $n, p$. 

Set $\mu=\lambda E$. Now we decompose the level set $E_\lambda (2N_1)$. By (\ref{ulambda definition}), (\ref{w estimate}) and (\ref{v D^2 estimate}),
$$\begin{aligned}
& \left|\left\{x \in B_{5 \rho_i}\left(x_i\right):\left|D^2 u\right|>2 N_1 \mu\right\}\right|=\left|\left\{x \in B_{5 \rho_i}\left(x_i\right):\left|D^2 u_\lambda\right|>2 N_1\right\}\right| \\
& \leq\left|\left\{x \in B_{5 \rho_i}\left(x_i\right):\left|D^2 w_\lambda\right|>N_1\right\}\right|+\left|\left\{x \in B_{5 \rho_i}\left(x_i\right):\left|D^2 v_\lambda\right|>N_1\right\}\right| \\
& =\left|\left\{x \in B_{5 \rho_i}\left(x_i\right):\left|D^2 w_\lambda\right|>N_1\right\}\right| \leq \frac{1}{N_1^p} \int_{B_{5 \rho_i}\left(x_i\right)}\left|D^2 w_\lambda\right|^p d x \leq \frac{C\left|B_{\rho_i}\left(x_i\right)\right|}{M^p} .
\end{aligned}
$$
By the estimate of $\left|B_{\rho_i}\left(x_i\right)\right|$ in Lemma 2.2.2 and \ref{ulambda definition} ,
$$
\begin{aligned}
& \left|\left\{x \in B_{5 \rho_i}\left(x_i\right):\left|D^2 u\right|>2 N_1 \mu\right\}\right| \\
& \leq \frac{C_1}{M^p \mu^p}\left(\int_{\left\{x \in B_{\rho_i}\left(x_i\right):\left|D^2 u\right|>\mu / 2\right\}}\left|D^2 u\right|^p d x+M^p \int_{\left\{x \in B_{\rho_i}\left(x_i\right):|f|>\mu /(2 M)\right\}}|f|^p d x\right),
\end{aligned}
$$
where $C_1=C_1(n, \phi)$. Note that the balls $\left\{B_{\rho_i}\left(x_i\right)\right\}$ are disjoint and
$$
\bigcup_{i \geq 1} B_{5 \rho_i}\left(x_i\right) \cup \text { negligible set } \supset E_\lambda(1)=\left\{x \in B_1:\left|D^2 u\right|>\mu\right\}
$$
for any $\lambda>0$, by replace $u$ by $u/2N_1$, 
$$
\begin{aligned}
& \left|\left\{x \in B_1:\left|D^2 u\right|>2 N_1 \mu\right\}\right| \leq \sum_i\left|\left\{x \in B_{5 \rho_i}\left(x_i\right):\left|D^2 u\right|>2 N_1 \mu\right\}\right| \\
& \leq \frac{C_1}{M^p \mu^p}\left(\int_{\left\{x \in B_1:\left|D^2 u\right|>\mu / 2\right\}}\left|D^2 u\right|^p d x+M^p \int_{\left\{x \in B_1:|f|>\mu /(2 M)\right\}}|f|^p d x\right) .
\end{aligned}
$$
Now, back to the very beginning of  this section 2.2.1, and usingLemma 2.2.3, 
$$
\begin{aligned}
\int_{B_1} \phi\left(\left|D^2 u\right|\right) d x= & \int_0^{\infty}\left|\left\{x \in B_1:\left|D^2 u\right|>2 N_1 \mu\right\}\right| d\left[\phi\left(2 N_1 \mu\right)\right] \\
\leq & \frac{C_1}{M^p} \int_0^{\infty} \frac{1}{\mu^p}\left\{\int_{\left\{x \in B_1:\left|D^2 u\right|>\mu / 2\right\}}\left|D^2 u\right|^p d x\right\} d\left[\phi\left(2 N_1 \mu\right)\right] \\
& +C_1 \int_0^{\infty} \frac{1}{\mu^p}\left\{\int_{\left\{x \in B_1:|f|>\mu /(2 M)\right\}}|f|^p d x\right\} d\left[\phi\left(2 N_1 \mu\right)\right] \\
\leq & \frac{C_2}{M^p} \int_{B_1} \phi\left(\left|D^2 u\right|\right) d x+C_3 \int_{B_1} \phi(|f|) d x
\end{aligned}
$$
where $C_2=C_2(n, \phi)$ and $C_3=C_3(n, \phi, M)$. Finally, choosing a suitable $M>0$ such that $C_2 / M^p<1 / 2$, we obtain
$$
\int_{B_1} \phi\left(\left|D^2 u\right|\right) d x \leq C \int_{B_1} \phi(|f|) d x .
$$

\end{proof}

\subsubsection{In the case when $f \in L^\phi$.} 

Now we use an approximation argument to complete the final proof.
\begin{proof}
    
Let $\left\{f_k\right\}_{k=1}^{\infty}$ be a sequence of smooth functions in $C_0^{\infty}\left(B_1\right)$ satisfying
$$
f_k \longrightarrow f \text { in } L^\phi\left(B_1\right)
$$
for a given Young function $\phi \in \triangle_2 \cap \nabla_2$. By the continuity of the convex $\phi$, it's easy to check that
\begin{equation} \label{approximation}
\int_{B_1} \phi\left(\left|f_k\right|\right) d x \longrightarrow \int_{B_1} \phi(|f|) dx .
\end{equation}

Now we consider the regularized problems
$$
-\Delta u_k=f_k \in C_0^{\infty}\left(B_1\right) \text { in } B_1.
$$
Because of the result in section 2.2.1,
$$
\int_{B_1} \phi\left(\left|D^2 u_k\right|\right) d x \leq C \int_{B_1} \phi\left(\left|f_k\right|\right) d x,
$$
where the constant $C$ is independent of $k \in \mathbb{N}$. Let $k \rightarrow \infty$. By (\ref{approximation}) and the lower semicontinuity of the left-hand side of above inequality, we obtain 
$$
\int_{B_1} \phi\left(\left|D^2 u\right|\right) d x \leq C \int_{B_1} \phi(|f|) d x .
$$

\end{proof}

\newpage
\bibliographystyle{ieeetr}
\small\bibliography{ref}

\end{document}